\documentclass[a4paper,12pt,reqno]{amsart}
\usepackage{amssymb}
\usepackage{amsmath}
\usepackage{enumitem}
\usepackage{mathrsfs}
\usepackage[all]{xy}
\usepackage{mathtools}
\usepackage{hyperref}
\usepackage{url}
\usepackage{bbm}

\usepackage{longtable,graphics,multirow}
\usepackage{tikz}
\usepackage{pgf,tikz}
\usetikzlibrary{arrows}
\usepackage[all]{xy}

\usepackage[OT2,T1]{fontenc}
\DeclareSymbolFont{cyrletters}{OT2}{wncyr}{m}{n}

\usepackage[margin=1.3in]{geometry}
\numberwithin{equation}{section} \numberwithin{figure}{section}

 \DeclareMathOperator{\rank}{rank}

\DeclareMathOperator{\HH}{H}

\DeclareMathOperator{\chr}{char}

\DeclareSymbolFont{cyrletters}{OT2}{wncyr}{m}{n}
\DeclareMathSymbol{\Sha}{\mathalpha}{cyrletters}{"58}
\DeclareMathSymbol{\Be}{\mathalpha}{cyrletters}{"42}

\newcommand{\OO}{\mathcal{O}}

\newcommand\PP{\mathbb{P}}

\newcommand\QQ{\mathbb{Q}}

\newcommand\GG{\mathbb{G}}
\newcommand\Ga{\GG_\mathrm{a}}

\newtheorem{lemma}{Lemma}

\newtheorem{theorem}[lemma]{Theorem}
\newtheorem{proposition}[lemma]{Proposition}
\newtheorem{corollary}[lemma]{Corollary}

\theoremstyle{definition}
\newtheorem{example}[lemma]{Example}
\newtheorem{definition}[lemma]{Definition}
\newtheorem{remark}[lemma]{Remark}

\numberwithin{lemma}{section}

\usepackage{color}

\begin{document}

\title{Rank jumps on elliptic surfaces and the Hilbert property}

\author{Daniel Loughran}
\address{
Department of Mathematical Sciences\\
University of Bath\\
Claverton Down\\
Bath\\
BA2 7AY\\
UK}
\urladdr{https://sites.google.com/site/danielloughran/}

\author{Cec\'{\i}lia Salgado}
\address{Instituto de Matem\'atica, Univ. Federal do Rio de Janeiro, Rio de Janeiro, Brazil}
\email{salgado@im.ufrj.br}
\urladdr{http://www.im.ufrj.br/\~{}salgado}
\subjclass[2010]
{14G05 (primary), 
14J27, 
11G05 
(secondary)}

\begin{abstract}
Given an elliptic surface over a number field, we study the collection of fibres whose Mordell--Weil rank is greater than the generic rank. Under suitable assumptions, we show that this collection is not thin. Our results apply to quadratic twist families  and del Pezzo surfaces of degree $1$.
\end{abstract}

\maketitle

\thispagestyle{empty}

\tableofcontents

\section{Introduction} \label{sec:intro}

Let $E$ be an elliptic curve over $\QQ$. By the classical Mordell--Weil theorem, the group of rational points $E(\QQ)$ of $E$ is a finitely generated abelian group. The rank of this group is an important invariant and it is an open problem whether the rank is uniformly bounded for all elliptic curves over $\QQ$.

A common method for producing elliptic curves of large rank is as follows: one takes an elliptic curve $E_t$ over the \emph{function field} $\QQ(t)$ with large rank (the generic rank), then all but finitely many specialisations $t \in \QQ$ will have rank at least as large of the generic rank. We push this method further in two ways: we can make the rank of a specialisation \emph{strictly larger} than the generic rank, and we can moreover achieve this for ``many'' choices of $t$ (namely for a so-called \emph{non-thin} set of $t$).  We now explain our results and setup in more detail.

\subsection{Statement of results}
Let $\pi:X \to \PP^1$ be an elliptic surface defined over a number field $k$, i.e.~a smooth projective surface endowed with a genus one fibration that admits a section. The presence of a section implies that all but finitely many fibres are elliptic curves. Moreover, a theorem of N\'eron \cite{Ne52} implies that the Mordell--Weil group of the generic fibre is finitely generated. 

Silverman's specialization theorem \cite[Thm.~C]{Sil} states that all but finitely many fibres have rank at least the generic rank. This result built on a theorem of N\'eron \cite[Thm.~6]{Ne52} (see also \cite[\S11.1]{Ser97})  over higher-dimensional  bases which says that outside of a \emph{thin set} of rational points, the rank is at least the generic rank. (We recall the definition of thin sets in \S\ref{sec:pre}.)
Given Silverman's result, a  natural question is whether one can construct infinitely many fibres whose rank is \emph{greater than} the generic rank. This has been achieved in various cases when $X$ is unirational or a K$3$ surface \cite{Billard, salgado1,salgado3,HindrySal}. 

In our paper we take N\'eron's theorem as our starting point, and prove results towards showing that the collection of fibres where the rank jumps is not a thin set. Our first result concerns elliptic surfaces with a bisection, i.e.~a geometrically integral curve $C \subset X$ such that the induced map to $\PP^1$ is finite of degree $2$.

\begin{theorem} \label{thm:conic}
	Let $\pi: X \to \PP^1$ be a geometrically rational elliptic surface over a number field $k$ with generic rank $r$. Assume that $\pi$ admits a bisection of arithmetic genus $0$ and that the generic fibre of $\pi$ is not a quadratic twist of a constant elliptic curve. Then the set
	$$\{ t \in \PP^1(k) : \rank X_t(k) \geq r + 2\}$$
	is not thin.
\end{theorem}

Informally,	Theorem \ref{thm:conic} says that there is no finite list of non-trivial polynomial conditions that contains the set of $t$ for which the rank of the fibre $X_t$ is at least $r+2$. Note that in the stated level of generality, it was not even known that the set under consideration is infinite.  Theorem \ref{thm:conic} seems particularly surprising in the light of a conjecture of Silverman \cite[p.~556]{Sil85}, which predicts that $100\%$ of $t$, when ordered by height, have rank $r$ or $r+1$.

Our hypothesis that $X$ admits a bisection of arithmetic genus $0$ is a natural one. The minimal models of a geometrically rational surface are either del Pezzo surfaces or conic bundle surfaces \cite[Thm.~1]{Isk79}, and in the latter case a smooth conic gives a bisection.

We  are also able to obtain results for quadratic twist families which are excluded by Theorem \ref{thm:conic}. Such surfaces have a very special form and can be completely described (see \S \ref{sec:Chatelet}). Our results here are as follows.

\begin{theorem} \label{thm:non_reduced}
	Let $f,g$ be separable non-constant polynomials with $\deg f = 3$
	and $\deg g \leq 2$. Let $X$ be a smooth projective model
	of the surface
	$$g(t) y^2 = f(x) \qquad \subset \mathbb{A}^3_k$$
	equipped with the elliptic fibration $\pi$ given by projecting to
	the $t$-coordinate.
	\begin{enumerate}
		\item The set
	$\{ t \in \PP^1(k) : \rank X_t(k) \geq 1\}$
	is not thin.
		\item If both $g$ and $f$ have a root over $k$, then
		 $\{ t \in \PP^1(k) : \rank X_t(k) \geq 2\}$ is not thin.
	\end{enumerate}
\end{theorem}

The surfaces in this theorem have generic rank $0$.
Taking $g(t) = t$, Theorem~\ref{thm:non_reduced} applies to the usual quadratic twist family of an elliptic curve. Here various explicit constructions of infinitely many quadratic twists with rank at least $2$ are known (see e.g.~\cite[\S8]{ST95} or \cite[\S3]{RS01}); however all these constructions produce a thin set of fibres, so do not allow one to prove part (2) of Theorem \ref{thm:non_reduced}.

The family in Theorem \ref{thm:non_reduced} has a bisection of arithmetic genus $0$, given by taking $x = x_0$ where $f(x_0) \neq 0$. Our theorems combined show that any geometrically rational elliptic surface with a bisection of arithmetic genus $0$ admits rank jumps on a non-thin set (see Theorem \ref{thm:once}).

Various examples to which our results apply can be found in \S\ref{sec:examples}; these include some rational elliptic surfaces and del Pezzo surfaces of degree $1$, where we generalise a result of  Koll\'ar and Mella \cite{KM}.

\subsection{Methodology and relation to the literature}
To prove our results, we show that any such bisection deforms into a family of bisections and gives rise to a conic bundle on $X$. This puts us into a position to apply the methods of \cite{salgado1}, where it is shown that all but finitely many of the conics increase the generic rank after base-change. Choosing such a conic with infinitely many rational points, Salgado used this \emph{ibid}.~to get infinitely many rank jumps for unirational elliptic surfaces. (Note that our surfaces are also unirational, cf.~Lemma \ref{lem:unirational}.)

However in our case we cannot just use a single conic, since the image of its rational points in $\PP^1(k)$ is clearly thin! We therefore have to consider all conics at once. The function field of each conic gives a quadratic extension of $k(\PP^1)$, and the key step in the proof is to show that we obtain infinitely many isomorphism classes of field extensions this way, whilst still guaranteeing that the conic has a rational point. For most $X$, this is relatively simple as the branch loci of the quadratic extension can be shown to move. However, the most difficult case turns out to be the quadratic twist family from Theorem  \ref{thm:non_reduced}. We perform a detailed analysis of the geometry of these surfaces in \S \ref{sec:Chatelet}, which have the distinguishing feature that there are two non-reduced fibres. For our work we need the observation that such surfaces are birational to Ch\^{a}telet surfaces, which allows us to use results of Colliot-Th\'{e}l\`{e}ne,  Sansuc, and Swinnerton-Dyer  \cite{CTSSD} on the Brauer--Manin obstruction to weak approximation for Ch\^{a}telet surfaces. This method is sufficient to obtain rank jump $1$ (Theorem \ref{thm:once}).

To get rank jump $2$  we push the method further, with again particular attention when there are non-reduced fibres. Note that Salgado also obtains a result \cite[Thm.~1.3]{salgado1} where it is shown that there are infinitely many fibres where the rank jumps at least twice; however her result contains technical assumptions (e.g.~$X$ is rational with two conic bundles and $\pi$ has at most one non-reduced fibre). It is exactly our more careful analysis of the more delicate case of quadratic twists which allows us to remove these technical assumptions, as well as the realisation that one can get rank jump $2$ by base-changing by the same conic bundle twice.

In \cite{Billard}, Billard obtains a similar result to \cite{salgado1} restricted to rational surfaces with a non-isotrivial elliptic fibration. His methods are quite different and make use of height theory. Assuming various standard conjectures, he shows in  \cite[p.~69]{Billard} that for every $\varepsilon >0$ there are $\gg B^{1- \varepsilon}$  points of (naive) height bounded by $B$ in $\PP^1(\QQ)$ for which the rank jumps. This is insufficient  to conclude that the subset under consideration is not thin, as the best known upper bound for points of height at most $B$ in a thin set is $O(B^{3/2}(\log B))$ \cite[Thm.~13.3]{Ser97}. Other related works include proofs of the Hilbert property for certain elliptic K3 surfaces \cite{CZ17,Dem20,Demeio}. These allow one to obtain rank jumps on a non-thin set in certain circumstances (e.g. if the generic rank is 0).

After seeing a talk by the second-named author on our work at the conference ``Rational points on Fano and similar varieties'' in Paris in May 2019, Colliot-Th\'el\`ene devised an alternative method to prove results about rank jumps and thin sets \cite{CT19}. He showed that the set of fibres where the rank jumps by at least $1$ is not thin, if $X$ satisfies so-called ``weak weak approximation''. This property is not known in our generality (e.g.~for the del Pezzo surfaces of degree $1$ considered in \S \ref{sec:DP1}). He is  able to obtain results about rank jump $2$, but only when the generic fibre has complex multiplication over $k$. So his results only recover the statement of Theorem \ref{thm:conic} in the very special case where $X$ satisfies weak weak approximation and the generic fibre has complex multiplication over $k$.

\subsection*{Notation}
Let $C$ be a projective curve (not necessarily integral). We define its arithmetic genus to be $p_a(C) = 1 - \chi(C,\OO_C)$. If $C$ is geometrically integral, we define its (geometric) genus  $g(C)$ of $C$ to be the genus of the normalisation of $C$.

Let $f:X \to Y$ be a morphism of varieties and $Z_1,Z_2 \subset X$ closed subvarieties. We denote by $Z_1 \times_f Z_2$ the fibre product of $Z_1$ and $Z_2$ which respect the morphisms $f_{|Z_i} :Z_i \to Y$.

\subsection*{Acknowledgements}
The authors would like to thank the Max Planck Institute for Mathematics in Bonn and the Institut Henri Poincar\'e in Paris for their hospitality. We are grateful to Sam Streeter and Jean-Louis Colliot-Th\'el\`ene for useful comments on the paper, and Matthias Schuett for helpful discussions on the arithmetic of elliptic fibrations. We also thank the referee for numerous useful remarks. The first-named author is supported by EPSRC grant EP/R021422/2. The second-named author is partially supported by FAPERJ grant E-26/203.205/2016,  the Serrapilheira Institute (grant number Serra-1709-17759) and the Capes-Humboldt program.

\section{Preliminaries} \label{sec:pre}
In this section we gather various definitions and basic results.
Let $k$ be a field.

\subsection{Elliptic surfaces}

\begin{definition}  \label{def:elliptic_surface}
An \emph{elliptic surface} over $k$ is a smooth projective surface $X$ together with a morphism $\pi: X \to B$ to some curve $B$ which admits a section and whose generic fibre is a smooth curve of genus $1$.

\end{definition}
We fix a choice of section to act as the identity element for each smooth fibre.
We say that $\pi$ is \emph{relatively minimal} if every birational morphism $f: X \to X'$ of elliptic surfaces over $k$ with $\pi' \circ f = \pi$ is an isomorphism.

We say that $X$ is \emph{constant} if its generic fibre is constant, i.e.~its generic fibre is the base change of some elliptic curve over $k$.

We call a geometrically integral curve $C \subset X$ a \emph{multisection} if the map $C \to B$ induced by $\pi$ is finite flat. We define the degree of $C$ to be the degree of the induced map. A bisection is a multisection of degree $2$. (By a \emph{curve} on a smooth projective surface  we mean an effective divisor, viewed as a closed subscheme.)

\begin{remark} \label{rem:relatively_minimal}
	An elliptic surface is relatively minimal over $k$ if and only if it is relatively
	minimal over $\bar{k}$. This follows from \cite[Cor.~9.3.24, Prop.~9.3.28]{Liu}.
\end{remark}

\begin{remark} \label{rem:anticanonical}
	Let $X\to \PP^1$ be a geometrically rational relatively minimal elliptic surface over $k$. Then the elliptic fibration on $X$ is unique and given by the anticanonical linear system $|-K_X|$ \cite[\S8.3]{ShiodaSchuett}.
 \end{remark}
 
To force the generic rank to jump, we use the following result of Salgado. In the statement, the base-changed surface $X \times_{\pi} D$ need no longer be regular, but its generic fibre is still an elliptic curve hence the generic rank is still well-defined.

\begin{theorem} \label{thm:pencil}
	Let $\pi:X \to B$ be an
	elliptic surface over a number field $k$  and 
	$L$ a pencil of curves on $X$ with an element
	which is a multisection of $\pi$.
	Assume that either $g(B) = 0$ or that $X$ is non-constant.
	Then for all but finitely many $D \in L$, the base-changed 
	surface $X \times_\pi D \to D$ has generic rank strictly larger
	than the generic rank of $X \to B$.
\end{theorem}
\begin{proof}
	If $X$ is non-constant then this is \cite[Cor.~4.3]{salgado1}.
	Otherwise, after a birational transformation, we may assume that
	$X = E \times B$ for some elliptic curve $E$ over $k$ and with $g(B) = 0$.
	As $g(B) = 0$ any map $B \to E$ is constant, thus any section of $\pi$
	is constant, i.e.~$X(B) = E(k)$. 
	Let $D \in L$ be a multisection. 
	Sections are $(-1)$-curves \cite[Cor.~6.9]{ShiodaSchuett}
	hence do not move in a linear system, thus $D$ is not a section.	
	The base-changed surface
	is the constant surface $X \times_{\PP^1} D = E \times D \to D$
	over $D$. This has a new section given by
	$$D \to X \times_{\PP^1} D, \quad x \mapsto (x,x),$$
	which is easily checked to be independent from the old sections.
\end{proof}

We make numerous uses of the following theorem of Silverman \cite{Sil}.

\begin{theorem} \label{thm:Silverman}
	Let $\pi:X \to B$ be an elliptic surface
	over a number field $k$ with generic rank $r$.  
	Assume that either $g(B) = 0$ or that $X$ is non-constant.
	Then the set
	$\{ t \in B(k) : \rank X_t(k) < r\}$
	is finite.
\end{theorem}
\begin{proof}
	If $X$ is non-constant then this is \cite[Thm.~C]{Sil}.
	Otherwise it suffices to treat the case 
	$X = E \times \PP^1$, where the result follows immediately from the 
	fact that any section is constant.
\end{proof}

\begin{remark}
	Theorem \ref{thm:Silverman} requires
	that the elliptic surface is not constant in general. For example,
		take $E$ an elliptic curve  of rank $1$
	and $X = E \times E \to E$ a constant surface over $E$.
	Then the fibre over every rational point has rank $1$,
	but the generic fibre has rank at least $2$, with a new section given by 
	$$E \to E \times E, \quad x \mapsto (x,x).$$
\end{remark}

Since we will use it numerous times, for ease of reference we recall the adjunction formula \cite[Thm.~9.1.37]{Liu}: For a smooth projective surface $X$ and curve $C \subset X$ we have
\begin{equation} \label{eqn:adjunction}
	2p_a(C) - 2 = C(C + K_X).
\end{equation}

\subsection{Conic bundles}

We use the following definition of conic bundles.
\begin{definition}\label{def:conic bundle}
	A \emph{conic bundle} on a smooth projective surface $X$
	is a dominant morphism $X \to \PP^1$ whose generic
	fibre is a smooth geometrically irreducible curve of genus $0$.
\end{definition}

Note that, contrary to some authors,  we do \emph{not} require that every fibre of a conic bundle is isomorphic to a plane conic.

\begin{lemma} \label{lem:conic_bundle}
	Let $X$ be a smooth projective
	surface over $k$ with $\HH^1(X,\OO_X) = 0$
	and $C \subset X$ a  curve of arithmetic genus $0$ 
	with $-K_X \cdot C = 2$. Then the complete linear system $|C|$ is a pencil.
	If moreover $C$ is geometrically integral then it induces a conic bundle on $X$.
\end{lemma}
\begin{proof}
	Consider the exact sequence of sheaves on $X$
	$$
		0 \to \OO_X \to \OO_X(C) \to \OO_C(C) \to 0.
	$$
	Applying cohomology gives 
	$$ 0 \to \HH^0(X, \OO_X) \to \HH^0(X, \OO_X(C)) \to \HH^0(X, \OO_C(C))
	\to \HH^1(X, \OO_X).$$
	The adjunction formula \eqref{eqn:adjunction}  implies
	$C^2=0$, thus $\dim \HH^0(X, \OO_C(C)) =1$. Using 
	our assumption $\HH^1(X, \OO_X)=0$, we find that $\dim |C| = 1$. 
	If $C$ is geometrically integral, as $C^2 = 0$ 
	the linear system $|C|$ is base-point-free and the generic
	member is  geometrically integral of arithmetic genus $0$,
	hence smooth \cite[Prop.~7.4.1]{Liu}.
\end{proof}

\begin{remark} \label{rem:fixed_component}
	Without the assumption that $C$ is geometrically integral, one need not
	obtain a conic bundle in general. Take $X \to \PP^1$ a rational elliptic surface
	and $C = -K_X + 2E$ where $E$ is a section. Then one checks that $p_a(C) = 0$
	and $-K_X \cdot C = 2$, but $|C|$ does not induce a conic bundle. 
	By Lemma \ref{lem:conic_bundle} one has
	$\dim |C| = 1$, but the problem
	is that $2E$ is a fixed component of $|C|$.
\end{remark}

\begin{lemma} \label{lem:strange}
	Let $X$ be a smooth projective surface over a field $k$ of characteristic
	$0$. Let $L$ be a base-point-free pencil on $X$ and $F \subset X$ a
	reduced curve. Then for all but finitely many $C \in L$,
	the (scheme-theoretic) intersection $C \cap F$ is reduced.
\end{lemma}
\begin{proof}
	Consider the morphism $X \to \PP^1$ induced by $L$ and its restriction
	$F \to \PP^1$ to $F$. As $\chr k = 0$ and $F$ is reduced,
	the generic fibre is smooth.
	It follows that the intersection of a general element of $L$
	with $F$ is reduced, as claimed.
\end{proof}

\begin{remark}
	Take $k$ an infinite field of characteristic $2$,
	and	consider a smooth plane conic $F \subset \PP^2_k$. Then $F$
	admits a ``strange point'', namely a point $P \in \PP^2_k$
	such that the pencil of lines $L$ through $P$ meets $F$ tangentially \cite[p.~76]{Sam}.
	Blowing up $P$ we obtain a base-point-free pencil
	which meets the strict transform of $F$ always in
	a point of multiplicity $2$,
	and thus does not satisfy the conclusion of Lemma \ref{lem:strange}.
\end{remark}

\subsection{Thin sets} We use Serre's definition of thin sets \cite[\S 3.1]{Ser08}.

\begin{definition}
Let $X$ be a variety over a field $k$. A subset $Z\subseteq X(k)$ is called \emph{thin}  if it is
a finite union of subsets which are either contained in a proper closed subvariety of $X$, or contained 
in some $\pi(Y(k))$ where $\pi: Y \to X$ is a generically finite dominant morphism 
of degree exceeding $1$, with $Y$ an integral variety over $k$.
\end{definition}

 A variety is said to satisfy the \emph{Hilbert property} if its rational points are not thin. Hilbert's irreducibility theorem implies that  projective space over a number field satisfies the Hilbert property \cite[Thm.~3.4.1]{Ser08}. It is an open problem in general whether every unirational variety satisfies the Hilbert property.

\begin{remark} \label{rem:P^1}
	To prove that a subset $Z \subseteq \PP^1(k)$ is not thin, it
	suffices to show the following: 
	For any finite collection of finite morphisms $Y_i \to \PP^1$ of
	degree at least $2$ with the $Y_i$ smooth geometrically integral
	curves, there exists $P \in Z$ which is not
	in the union of the images of the maps $Y_i(k) \to \PP^1(k)$.
\end{remark}

\section{Elliptic surfaces with two non-reduced fibres} \label{sec:Chatelet}
In our proofs, we will need to treat carefully elliptic surfaces with 2 non-reduced fibres. We gather in this section the  geometric facts we will require about such surfaces. We work over a field $k$ of characteristic $0$ with algebraic closure $\bar{k}$.

\begin{proposition}\label{prop:twononred}
 Let $\pi:X \to \PP^1$ be a geometrically rational relatively minimal elliptic surface over $k$. Then the following are equivalent.
 \begin{enumerate}
 	\item The generic fibre of $X$ is a quadratic twist of a constant elliptic
 	curve.
 	\item There is a Weierstrass equation for $X$ of the form $g(t)y^2=f(x)$ over $k$,
 	where $f,g$ are separable non-constant polynomials with $\deg f = 3$
	and $\deg g \leq 2$. 
	\item The singular fibre configuration  of $\pi$ over $\bar{k}$ is $2I_0^*$.
	\item The map $\pi$ has more than one non-reduced fibre over $\bar{k}$.
 \end{enumerate}
\end{proposition}
\begin{proof}
	We first show that $(1)$ implies $(2)$. Our assumptions
	imply that the generic fibre of $X$ has an equation
	of the form	$g(t)y^2 = f(x)$ for some $g \in k(t)$.
	After making a change of variables, we may assume that $g \in k[t]$ and that $g$
	is separable. We obtain the Weierstrass equation
	\begin{equation} \label{eqn:W}
		y^2 = x^3 + a g(t)^2 x + bg(t)^3
	\end{equation}
	for some $a,b \in k$. As $g$ is separable, 
	the resulting equation is easily verified to be globally
	minimal \cite[\S4.8]{ShiodaSchuett}.
	However as $X$ is geometrically rational
	$g(t)$ must be  non-constant
	and $\deg g(t) \leq 2$ (see \cite[\S8.4]{ShiodaSchuett})
	
	
	That $(2)$ implies $(3)$ follows from a simple application
	of Tate's algorithm \cite[\S IV.9.4]{Sil94},
	whilst $(3)$ implies $(4)$ is trivial.
	
	We now prove that $(4)$ implies $(3)$. This can be shown by 
	inspecting the list 
	of configurations  in \cite[p.~7-13]{Persson}. Alternatively, choose
	a globally minimal Weierstrass equation  with discriminant
	$\Delta(t)$. Then a root of $\Delta$ corresponds to a singular fibre,
	and a non-reduced fibre gives rise to a root of multiplicity at least 6, with equality
	only for $I_0^*$ (see Table 4.1 in \cite[\S IV.9.4]{Sil94}). 
	As $\deg \Delta(t) \leq 12$ \cite[\S IV.8.4]{Sil94},
	the only possibility is  $2I_0^*$.
	
	For $(3)$ implies $(2)$, without loss of generality, we may assume 
	that the fibre at infinity is smooth. So let $g(t)$ be a separable
	quadratic polynomial whose zero locus lies below the non-reduced fibres.
	Then $X$ admits a Weierstrass equation $y^2=x^3+A(t)x+B(t)$ with
	$\deg (A(t))\leq 4$ and $\deg (B(t))\leq 6$ \cite[\S8.4]{ShiodaSchuett},
	and for some constant $c$ the discriminant satisfies 
	\begin{equation} \label{eqn:Delta}
	 \Delta(t)= -16(4A(t)^3+27B(t)^2)= cg(t)^6.
	 \end{equation}
 	Moreover $g(t)$ divides $A(t)$ as the singular fibres are additive 
 	(see Tate's algorithm \cite[\S IV.9.4]{Sil94}),  
 	hence also divides $B(t)$. From \eqref{eqn:Delta}, 
 	one easily deduces that there are constants $a$ and $b$ such that
 	$A(t)=ag(t)^2$ and $B(t)=bg(t)^3$. We then obtain an equation of the form
 	\eqref{eqn:W}, and a change of variables gives (2).
 	Finally, $(2)$ trivially implies $(1)$.
\end{proof}

We now let $\pi: X \to \PP^1$ be as in Proposition \ref{prop:twononred}. Recall that by a \emph{$(-n)$-curve} (for $n > 0$) on a smooth projective surface $S$, we mean a geometrically integral curve $C \subset S$ of arithmetic genus $0$ with $C^2 = -n$.

\begin{lemma} \label{lem:figure}
	The intersection graph of negative curves on $X_{\bar{k}}$ is given by Figure~\ref{exceptionalcurves}. The $T_i$ are $2$-torsion sections of $\pi$ and the two singular fibres of $\pi$ are given by $F_i=2E_i + L_{i,1} + L_{i,2} + L_{i,3} + L_{i,4}$. 
	
\begin{figure}[h]
\begin{tikzpicture}[thick,scale=1.4, every node/.style={scale=1}]
\draw (5,0) node{$\bullet$};
\draw (5,0) node[above]{\small{$E_1\,\,\,$}};
\draw (5,0) -- (6,1);
\draw (5,0) -- (6,0.35);
\draw (5,0) -- (6,-0.3);
\draw (5,0) -- (6,-1);
\draw (6,1) node{$\bullet$};
\draw (6,1) node[above]{\small{$L_{1,1}$}};
\draw (6,0.35) node{$\bullet$};
\draw (6,0.35) node[above]{\small{$L_{1,2}$}};
\draw (6,-0.3) node{$\bullet$};
\draw (6,-0.3) node[above]{\small{$L_{1,3}$}};
\draw (6,-1) node{$\bullet$};
\draw (6,-1) node[above]{\small{$L_{1,4}$}};
\draw (6,1) -- (7,1);
\draw (6,0.35) -- (7,0.35);
\draw (6,-1) -- (7,-1);
\draw (6,-0.3) -- (7,-0.3);
\draw (7,1) -- (8,1);
\draw (7,0.35) -- (8,0.35);
\draw (7,-1) -- (8,-1);
\draw (7,-0.3) -- (8,-0.3);
\draw[fill=white] (7,1) circle (2pt);
\draw (7,1) node[above]{\small{$T_1$}};
\draw[fill=white] (7,0.35) circle (2pt);
\draw (7,0.35) node[above]{\small{$T_2$}};
\draw[fill=white] (7,-0.3) circle (2pt);
\draw (7,-0.3) node[above]{\small{$T_3$}};
\draw[fill=white] (7,-1) circle (2pt);
\draw (7,-1) node[above]{\small{$T_4$}};
\draw (8,1) node{$\bullet$};
\draw (8,1) node[above]{\small{$L_{2,1}$}};
\draw (8,0.35) node{$\bullet$};
\draw (8,0.35) node[above]{\small{$L_{2,2}$}};
\draw (8,-0.3) node{$\bullet$};
\draw (8,-0.3) node[above]{\small{$L_{2,3}$}};
\draw (8,-1) node{$\bullet$};
\draw (8,-1) node[above]{\small{$L_{2,4}$}};
\draw (8,1) -- (9,0);
\draw (8,0.35) -- (9,0);
\draw (8,-1) -- (9,0);
\draw (8,-0.3) -- (9,0);
\draw (9,0) node[above]{\small{$\,\,\,E_2$}};
\draw (9,0) node{$\bullet$};
\end{tikzpicture}
\caption{The configuration of negative curves on $X_{\bar{k}}$. \newline We denote by $\bullet$ a $(-2)$-curve and $\circ$ a $(-1)$-curve. }
\label{exceptionalcurves}
\end{figure}
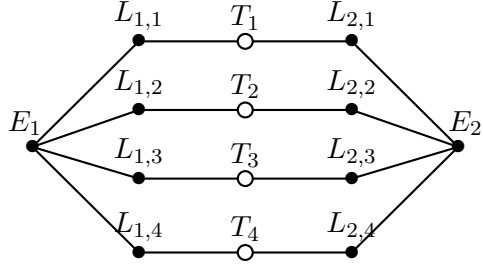
\end{lemma}
\begin{proof}
Let $C$ be a $(-n)$-curve. The adjunction formula \eqref{eqn:adjunction} gives $-K_C \cdot C = n - 2$. As $-K_X$ is nef (Remark \ref{rem:anticanonical}), we have $-K_X \cdot C \geq 0$, hence $n \geq 2$. We also find that $(-1)$-curves correspond to sections of the elliptic fibration, while $(-2)$-curves are components of singular fibres.  Since the N\'eron-Severi group of $X$ has rank 10 \cite[\S8.8]{ShiodaSchuett}, the Shioda--Tate formula \cite[Cor.~6.13]{ShiodaSchuett} implies that the generic rank is $0$, hence $X$ has only finitely many negative curves. For the surfaces we consider, there are four 2-torsion sections (see the bottom of p.~8 of the list in \cite{Persson}) and ten reducible components of singular fibres. To derive the intersection behaviour in Figure \ref{exceptionalcurves}, note that torsion sections are disjoint \cite[Lem.~1.1]{MirandaPersson}, and thus different 2-torsion sections intersect different components of the $I_0^*$ fibre in a smooth point (the connected component being $\Ga$, hence torsion free \cite[\S IV.9.4, Tab.~4.1]{Sil94}).
\end{proof}

\begin{lemma} \label{lem:Chatelet}
	The surface $X$ is birational to a Ch\^{a}telet surface.
	If $k$ is a number field then $X$ satisfies the Hilbert property.
\end{lemma}
\begin{proof}
	We make a change of variables to write $g(t) = t^2 -a$ 
	for some $a \in k$ and obtain
	$y^2t^2 - ay^2 = f(x).$
	The change of variables $w = yt$ then gives
	$$w^2 - ay^2 = f(x)$$
	which is  a Ch\^{a}telet surface.
	These satisfy so-called weak weak approximation by
	\cite[Thm.~B]{CTSSD}, hence satisfy the Hilbert property
	by \cite[Thm.~3.5.7]{Ser08}.
\end{proof}

By the previous lemma  $X$ in fact always has a conic bundle structure (this is given by projecting onto $x$ in the equation in Proposition \ref{prop:twononred}). We call this the \emph{Ch\^{a}telet bundle} on $X$.

\begin{proposition}  \label{prop:Chatelet}
	\hfill
	\begin{enumerate}
		\item Each conic in the Ch\^{a}telet bundle  meets the 
		non-reduced fibres $F_1$ and $F_2$ of $\pi$ over $\bar{k}$
		in the non-reduced components  $2E_1$ and $2E_2$, respectively.
		\item Assume that the $F_i$ are defined over $k$ 
		and that $\pi$ has non-trivial $2$-torsion.
		Then there is a conic bundle on $X$ whose smooth fibres
		meet the non-reduced fibres $F_i$ of $\pi$ in a 
		reduced component.
	\end{enumerate}	
\end{proposition}
\begin{proof}
(1) We use the curves from Figure \ref{exceptionalcurves}. One obtains the Ch\^{a}telet surface by blowing down the $T_i$. The singular fibres of the conic bundle on the Ch\^{a}telet surface are  the images of  $L_{1,i} + L_{2,i}$, and the curves $E_i$ thus give sections of the conic bundle over $\bar{k}$. Pulling back, the singular fibres of the Ch\^{a}telet bundle on $X$ are given by $L_{1,i} + 2T_i + L_{2,i}$. We have $E_i \cdot (L_{1,i} + 2T_i + L_{2,i}) = 1$, thus each conic meets $F_i$ in  $E_i$, as claimed.

(2) 
Our assumptions imply that two  sections are defined over $k$, say $T_1,T_2$. That the $F_i$ are defined over $k$ implies that the $E_i$ are also defined over $k$. Consider the curve $C = E_1 + L_{1,1} + L_{1,2} + T_1 + T_2$, which is defined over $k$. One easily checks from Figure \ref{exceptionalcurves} that $C^2 = 0$ and $-K_X \cdot C = 2$, thus by \eqref{eqn:adjunction} and Lemma \ref{lem:conic_bundle} this moves in a pencil. We will prove that this yields the required conic bundle (we have to be careful here, cf.~Remark \ref{rem:fixed_component}).

We first claim that $|C|$ has no fixed component. Write $C = D + F$ where $D,F$ are effective divisors such that $|D|$ has no  fixed component. Then we can write
$$D = a_1 E_1 + a_2 L_{1,1} + a_3 L_{1,2} + a_4 T_1 + a_5 T_2, \quad a_i \in \{0,1\}.$$
From Figure \ref{exceptionalcurves} we have
\begin{equation} \label{eqn:quadratic_form}
	D^2 = -2a_1^2 -2a_2^2 - 2a_3^2 -a_4^2 - a_5^2 + 2(a_1a_2 + a_1a_3 + a_2a_4 + a_3a_5).
\end{equation}
Let $M$ be the real symmetric matrix underlying the quadratic form in \eqref{eqn:quadratic_form}. Then $M$ has $5$ distinct eigenvalues,
 of which $4$ are negative, and one is equal to $0$ with eigenvector $(1,1,1,1,1)$. We deduce that $D^2 \leq 0$ with equality if and only if $\mathbf{a} = \mathbf{0}$ or $\mathbf{1}$. The case $D^2 < 0$ cannot occur as $|D|$ is a pencil without fixed component. Moreover if $\mathbf{a} = \mathbf{0}$ then $|D|$ is not a pencil, so we must have $\mathbf{a} = \mathbf{1}$. This implies that $F=0$ hence $|C|$ has no fixed component, as claimed.

As $C^2 = 0$ and $|C|$ is a pencil without fixed components, we deduce that $|C|$ is base-point-free. Whence by generic smoothness, the generic member is geometrically integral smooth of arithmetic genus $0$, hence a conic, so $|C|$ induces a conic bundle. From Figure \ref{exceptionalcurves} we find that $C \cdot E_i = 0$, hence the smooth members of the conic bundle avoid the non-reduced fibres of $\pi$, as required.
\end{proof}

\section{Rank jump once}

We now begin the proof of our main results. Our first aim is to prove Theorem \ref{thm:non_reduced}.(1). To prepare for the proof of Theorem \ref{thm:conic} and also illustrate the additional difficulties for quadratic twist families, we prove the following general result.

\begin{theorem} \label{thm:once}
	Let $\pi: X \to \PP^1$ be a geometrically rational elliptic surface over a number field $k$ with generic rank $r$. Assume that $\pi$ admits a bisection of arithmetic genus $0$. Then the set
	$\{ t \in \PP^1(k) : \rank X_t(k) \geq r + 1\}$
	is not thin.
\end{theorem}

We note that the surface $X$ in Theorem \ref{thm:non_reduced} satisfies the assumptions of Theorem \ref{thm:once}, as there is always a family of bisections given by the Ch\^{a}telet bundle (explicitly, take the curve $x=x_0$ for some $x_0 \in k$ with $f(x_0) \neq 0$).

\subsection{First steps}
To prove Theorem \ref{thm:once} we may assume that $\pi:X \to \PP^1$ is relatively minimal. Let $C \subset X$ be the bisection of arithmetic genus $0$. By Remark \ref{rem:relatively_minimal}, we see that $\pi$ is still relatively minimal over the algebraic closure. In particular, the elliptic fibration is given by the anticanonical class $-K_X$ (Remark \ref{rem:anticanonical}), thus $C \cdot (-K_X) = 2$. Hence by Lemma~\ref{lem:conic_bundle}, we find that $L:=|C|$ is a base-point-free pencil which gives rise to a conic bundle $\psi:X \to \PP^1$.

\begin{lemma} \label{lem:unirational}
	The surface $X$ is unirational, hence
	$X(k)$ is Zariski-dense.
\end{lemma}
\begin{proof}
	By Remark \ref{rem:anticanonical} we have $K_X^2 = 0$.
	Let $E \subset X$ be the zero section of $\pi$.

	As $C$ is nef we have $C \cdot E \geq 0$. 	
	First suppose
	that $C \cdot E = 0$. In this case $E$ is a $(-1)$-curve and a component
	of a fibre of the conic bundle. Contracting $E$, we obtain a conic bundle 
	surface $X'$ with $K_{X'}^2 = 1$ and $X'(k) \neq \emptyset$.
	But $X'$ is unirational by \cite[Thm.~7]{KM},
	hence $X$ is unirational.

	Now suppose that $C \cdot E > 0$. Then $E$ is a \emph{rational}
	multisection of the conic bundle $\psi$.
	Then $X \times_{\psi} E \to E$ is a conic bundle surface with
	a section, hence is rational (cf.~\cite[Lem.~23]{KM}).
	As this dominates $X$, we deduce $X$ is unirational.
\end{proof}

\subsection{Quadratic extensions}

If a conic $D \in L$ is integral, then the restriction $\pi|_D : D \to \PP^1$ is a degree $2$ cover; thus $k(\PP^1) \subset k(D)$ is a quadratic extension
which ramifies in exactly  $2$ points. The key result we require is the following.

\begin{proposition} \label{prop:quadratic}
	The set
	$$\{ k(\PP^1) \subset k(D) : D \in L, D(k) \neq \emptyset\}$$
	contains infinitely many isomorphism classes of quadratic extensions.
\end{proposition}

To prove this, we require the following on the ramification of $k(\PP^1) \subset k(D)$.

\begin{lemma} \label{lem:ramification}
	Let $P \in \PP^1$ lie below a reduced fibre $F$ of $\pi$.
	Then there are only finitely many
	$D \in L$ such that $k(\PP^1) \subset k(D)$ ramifies at $P$.
\end{lemma}
\begin{proof}
	We have that 
	$k(D)$ ramifies over $P$ if and only if $F \cap D$ is non-reduced.
	However by Lemma \ref{lem:strange} this happens for only finitely many
	$D$, as required.
\end{proof}

\begin{proof}[Proof of Proposition \ref{prop:quadratic}]
We divide the proof into two cases, depending on the number of reduced fibres of $\pi$.

\emph{At most one non-reduced fibre:}
Lemma \ref{lem:ramification} shows that at most one branch point of $D$ can be fixed and the other point must move in general. But quadratic extensions of $k(\PP^1)$ with different branch loci are non-isomorphic, hence given a quadratic extension $k(\PP^1) \subset K$, there are only finitely many $D$ with $k(D) \cong K$. Moreover by Lemma \ref{lem:unirational}, there are infinitely many $D$ with $D(k) \neq \emptyset$. Proposition \ref{prop:quadratic} now easily follows in this case.

\emph{Two non-reduced fibres:}
This is the more delicate case.
By Proposition \ref{prop:twononred} our surface has an equation of the shape
$$g(t)y^2 = f(x)$$
where $\deg f = 3$ and $\deg g = 2$ are separable. The elliptic fibration  $\pi$ is given by projecting to $t$. We will use the Ch\^{a}telet conic bundle $\psi$ given by projecting to $x$ (see \S\ref{sec:Chatelet}), since there is no reason in general to expect the existence of another conic bundle.
The quadratic extensions for $x \in k$ with $f(x) \neq 0$ are $$k(t) \subset k(t)(\sqrt{f(x)g(t)}).$$
This ramifies exactly over the roots of $g(t)$; this is independent of $x$ so the conclusion of Lemma \ref{lem:ramification} does not hold in this case, hence we must take a different approach. 
Assume that the set in Proposition \ref{prop:quadratic} is finite. Then as $(x,y,t)$ varies over the rational points of $X$, it follows that $f(x)$ takes only finitely many values modulo squares. Therefore, for $d \in k^*$, it suffices to show that the set 
\begin{equation}  \label{eqn:squares}
	\{ (x,y,t) \in X(k): f(x)  = d w^2, \mbox{ for some } w \in k  \}
\end{equation}
is thin, since this contradicts that $X$ satisfies 
the Hilbert property (Lemma~\ref{lem:Chatelet}).
But consider the geometrically integral curve $Y$ with the map
$$Y: f(x)= dw^2 \subset \mathbb{A}^2, \quad Y \to \PP^1, \, (x,w) \mapsto x.$$
Then we have the base-change with respect to the conic bundle $\psi:X \to \PP^1$
$$Z:=X \times_{\PP^1} Y \to X.$$
The image of $Z(k) \to X(k)$ is exactly the set \eqref{eqn:squares}. As the generic fibre of $X \to \PP^1$ is geometrically irreducible, it easily follows that $Z$ is again irreducible. Hence it follows from the definition of thin sets that \eqref{eqn:squares} is thin, as required.
\end{proof}

\subsection{Completion of the proof} \label{sec:complete}
We now prove Theorem \ref{thm:once}. Let $Y_i \to \PP^1$ be a finite collection of finite morphisms of degree at least $2$ with the $Y_i$ geometrically integral. 
By Theorem \ref{thm:pencil}, all but finitely many of the conics $D \in L$ are bisections that increase the generic rank after base-change.
The field extensions $k(\PP^1) \subset k(Y_i)$ have only finitely many quadratic subfields. Proposition \ref{prop:quadratic} therefore implies that there are infinitely many $D \in L$ such that
\begin{enumerate}
	\item $D(k) \neq \emptyset$;
	\item $k(D)$ is not a subfield of any $k(Y_i)$;
	\item The generic fibre of 	$X \times_{\pi} D \to D$
	has Mordell--Weil rank at least $r+1$.
\end{enumerate}
Choose such a $D$. As quadratic extensions are Galois, it follows from (2) that $k(D)$ is linearly disjoint with each $k(Y_i)$, so the curves $D \times_{\PP^1} Y_i$ are integral. By (1) we have $D \cong \PP^1$, so $D$ satisfies the Hilbert property. Thus there are infinitely many   $x \in D(k)$ not in the union of the images  of the  $(D\times_{\PP^1} Y_i)(k) \rightarrow D(k)$. Then the point $\pi(x) \in \PP^1(k)$ does not lie in the image of any $Y_i(k) \to \PP^1(k)$. However by (3) and Theorem \ref{thm:Silverman} applied to $X \times_{\pi} D \to D$, the  rank of $\pi^{-1}(x)$  is at least  $r+1$ for infinitely many $x$. This shows that the set in Theorem \ref{thm:once} is not contained in any thin subset, as required (cf.~Remark~\ref{rem:P^1}). \qed

\begin{remark}
	In the case of two non-reduced fibres, the elliptic surface becomes
	constant after base changing to $D$. Nonetheless $D \cong \PP^1$,
	so Silverman's theorem (Theorem \ref{thm:Silverman}) still applies
	in this case.
\end{remark}

\section{Rank jump twice}
In this section we give a uniform proof of Theorem \ref{thm:conic} and part (2) of Theorem~\ref{thm:non_reduced}, by building on the proof of Theorem \ref{thm:once}. 

\subsection{Biquadratic extensions}
Let $\pi: X \to \PP^1$ be a geometrically rational relatively minimal elliptic surface over a number field $k$ of generic rank $r$ with a bisection of arithmetic genus $0$. As in the proof of Theorem \ref{thm:once}, this induces a conic bundle structure $\psi:X \to \PP^1$. If $X$ is as in Theorem~\ref{thm:non_reduced}.(2), then by Proposition \ref{prop:Chatelet} it admits a conic bundle which is not the Ch\^{a}telet bundle; we choose $\psi$ to be such a bundle satisfying Proposition \ref{prop:Chatelet}.(2).

Let $Y_i \to \PP^1$ be a finite collection for $i \in I$ of finite morphisms of degree at least $2$ with the $Y_i$ smooth geometrically integral curves (cf.~Remark~\ref{rem:P^1}). We next  view the fields $k(Y_i)$  and the function fields of the conics as extensions of $k(\PP^1)$ and compare their branch points.

\begin{proposition} \label{prop:2}
	Let $\mathcal{P}=\{P_1,\dots,P_r\}$ for some $P_j \in \PP^1_k$.
	There exist infinitely many 
	$D_1$ and infinitely many $D_2$ which are fibres of the conic bundle
	$\psi$, such that:
	\begin{enumerate}
	\item $D:=D_1 \times_{\pi} D_2$ is a geometrically integral curve with a Zariski-dense set of rational points;
	\item $k(D_1) \otimes_{k(\PP^1)} k(D_2)$ 
	is linearly disjoint with every $k(Y_i)$;
	\item The generic fibre of 	$X \times_{\pi} D \to D$
	has Mordell--Weil rank at least $r+2$.
	\item Either 
	\begin{enumerate}
		\item $g(D) = 0$ and $D$ ramifies above at most one of the $P_j$, or
		\item $g(D)=1$ and $D$ does not ramify above any of the $P_j$.
	\end{enumerate}
\end{enumerate}
\end{proposition}
\begin{proof}
To prove the result, we are free to increase $\mathcal{P}$. We therefore assume that $\mathcal{P}$ contains all the branch points of the $Y_i$ and the singular locus of $\pi$.
We first assume that every fibre of $\pi$ is reduced. Then we can vary the ramification of $D_1$ using Lemma \ref{lem:ramification}. As in \S\ref{sec:complete}, we deduce that there are infinitely many $D_1$ with $D_1(k) \neq \emptyset$ such that $D_1$ is not ramified above any of the $P_j$, and	$X_{D_1}:=X \times_{\pi} D_1 \to D_1$ is an elliptic surface with generic rank at least $r+1$. We fix such a $D_1$ and next  construct $D_2$. 
Note that $X_{D_1}(k)$ is Zariski-dense as $X_{D_1}$ has an elliptic fibration of positive rank.

Composing with $\psi$ gives a map $\psi': X_{D_1} \to X \xrightarrow{\psi} \PP^1$ whose fibres are the curves $D_1 \times_{\pi} D_2$ as $D_2$ runs over the fibres of $\psi$. The generic fibre of $\psi'$ is a smooth curve of genus $1$, as only finitely many $D_2$ share a branch point with $D_1$ by Lemma \ref{lem:ramification}. These $D = D_1 \times_{\pi} D_2$ are a family of bisections of the elliptic fibration $X_{D_1} \to D_1$ induced by $\pi$, hence by Theorem \ref{thm:pencil} all but finitely many $D_2$ satisfy $(3)$. As $X_{D_1}(k)$ is Zariski-dense, it follows from \cite[Lem.~3.2]{Demeio} that $\psi'$ has infinitely many fibres with infinitely many rational points, which shows (1). 
As there are infinitely many such $D_2$, we can arrange (4b) using Lemma \ref{lem:ramification}. Next $k(\PP^1)\subset k(D)$ is a biquadratic extension,  hence Galois. To show (2), it suffices to note that the $k(Y_i)$ contain none of the quadratic subfields of $k(D)$ by (4b) and our choice of $\mathcal{P}$. This completes the proof in this case.

If $\pi$ has two non-reduced fibres over $\bar{k}$, then, by our choice of $\psi$ (Proposition~\ref{prop:Chatelet}), the smooth fibres of $\psi$ do not meet meet a non-reduced component of $\pi$, hence as in the proof of Lemma~\ref{lem:ramification} we can move both branch points of the fibres of $\psi$. The proof now proceeds exactly as in the previous case in which every fibre of $\pi$ is reduced.

Finally assume that $\pi$ has one non-reduced fibre over $\bar{k}$, which without loss of generality  lies over $P_1$  (this is necessarily a rational point). If there is at least one smooth fibre of $\psi$ which is unramified over $P_1$, then we can vary the ramification and the proof is similar to the above. So we assume that all smooth fibres of $\psi$ ramify over $P_1$. As in the proof of Proposition \ref{prop:quadratic}, we find that there are infinitely fibres $D_1$ of $\psi$ with $D_1(k) \neq \emptyset$ such that the intersection of the branch locus of $D_1$ and $\mathcal{P}$ is exactly $P_1$, and such that $X_{D_1}:=X \times_{\pi} D_1 \to D_1$ has generic rank at least $r+1$. (Note that $X_{D_1}$ is singular here, so not strictly an elliptic surface in the sense of Definition \ref{def:elliptic_surface}). We fix such a $D_1$ and next  construct $D_2$. 

We again consider the map $\psi':X_{D_1} \to \PP^1$ whose fibres are the curves $D=D_1 \times_{\pi} D_2$ as $D_2$ runs over the fibres of $\psi$. Again $D$ are a family of bisections of the elliptic fibration induced by $\pi$, hence applying Theorem \ref{thm:pencil} to a desingularisation of $X_{D_1}$, we see that all but finitely many $D_2$ satisfy $(3)$. Here $D_1$ and $D_2$ both ramify over $P_1$, and by Lemma \ref{lem:ramification} share no other branch points in general. Now (2) immediately follows, since any quadratic subfield of $k(D)$ ramifies over a point outside of $\mathcal{P}$, so cannot be a subfield of any of the $k(Y_i)$. Moreover Lemma \ref{lem:genus_0} implies (4b) and that $D$ is geometrically integral. To conclude, as $X_{D_1}(k)$ is Zariski-dense and the generic fibre of $\psi'$ has genus $0$, there are infinitely many $D_2$ such that $D(k)$ is Zariski-dense, which shows (1).
\end{proof}

\begin{lemma} \label{lem:genus_0}
	Let $C_1,C_2$ be smooth projective geometrically integral curves of genus $0$
	over a field $K$ of characteristic $0$ equipped with  morphisms
	$f_i: C_i \to \PP^1$ of degree $2$ that have exactly one branch
	point in common. Then $C_1 \times_{\PP^1} C_2$ is  a geometrically integral 
	curve of genus $0$.
\end{lemma}
\begin{proof}
	Our hypotheses imply that $K(C_1) \not \cong K(C_2)$ as quadratic extensions of $K(\PP^1)$,
	hence they are linearly disjoint. This implies that $C:=C_1 \times_{\PP^1} C_2$ 
	is geometrically integral. To calculate the genus, we may pass to the algebraic closure
	and moreover assume that the branch points are $0,1,\infty$. Thus we have
	$$K(C_1) \cong K(t)(\sqrt{t}), \quad K(C_2) \cong K(t)(\sqrt{t(t-1)}).$$
	An affine patch of $C$ therefore has the equations
	$$z^2 = t, \quad w^2 = t(t-1).$$
	Rearranging gives $w^2 = z^2(z^2 - 1)$, which defines a curve of genus $0$.
\end{proof}

\begin{remark}
	It is essential for the proof of Proposition \ref{prop:2} 
	in the case Theorem \ref{thm:non_reduced}.(2) that  $\psi$
	is not the Ch\^{a}telet bundle.
	Otherwise, any two conics $D_1$ and $D_2$ ramify over exactly
	the same points; here $D_1 \times_\pi D_2$
	is never geometrically integral and 
	Proposition \ref{prop:2}.(1) fails.
\end{remark}

\subsection{Proof of Theorem \ref{thm:conic} and Theorem \ref{thm:non_reduced}.(2)}
We choose $D$ as in Proposition \ref{prop:2}, with $\mathcal{P}$ the union of all the branch points of the $Y_i$ and the singular locus of $\pi$. If $g(D) = 0$, then the proof is essentially the same as the proof of Theorem \ref{thm:once}; as $D$ has the Hilbert property there are infinitely many $P \in \PP^1(k)$ in the image of $D(k)$ not in the image of any of the $Y_i(k)$. By Proposition \ref{prop:2} and Theorem \ref{thm:Silverman} applied to a desingularisation of $X \times_{\pi} D \to D$, we find that for all but finitely many choices of $P$ the rank of the fibre jumps twice, as required.

We therefore consider the case $g(D) = 1$. Here the previous argument breaks down, as $D$ does not satisfy the Hilbert property. Nonetheless, we have the following.

\begin{lemma} \label{lem:finite}
	The set $(D \times_{\PP^1} Y_i)(k)$ is finite for all $i$.
\end{lemma}
\begin{proof}
By construction, the curves $D_1, D_2, Y_i$ are all smooth and share no branch points in common. Therefore $D \times_{\PP^1} Y_i$ is geometrically integral and smooth (see e.g.~\cite[Lem.~2.8]{Str}). But the map $D \times_{\PP^1} Y_i \to D$ is ramified somewhere above one of the ramification points of $Y_i$. As $g(D) = 1$, Riemann--Hurwitz implies that  $g(D \times_{\PP^1} Y_i) \geq 2$. The result now follows from Faltings's theorem \cite{Fal83}.
\end{proof}

The proof of rank jump $2$ now runs in a similar way to the case $g(D) = 0$. Namely, by Proposition \ref{prop:2} and Lemma \ref{lem:finite}, there are infinitely many rational points $P \in \PP^1(k)$  which do not lie in the image of any of the $Y_i(k)$ but do lie in the image of $D(k)$. We then wish to apply Theorem \ref{thm:Silverman} to $X \times_{\pi} D \to D$. But as $g(D) = 1$, to do so we need to know that this surface is not constant.

\begin{lemma}
	The elliptic surface $X \times_{\pi} D \to D$ is non-constant.
\end{lemma}
\begin{proof}
	The original elliptic surface $X \to \PP^1$ is relatively minimal
	and geometrically rational; in particular it has at least one
	singular fibre \cite[\S8.4]{ShiodaSchuett}. But by construction,
	$D$ ramifies away from singular fibres of $\pi$.
	Thus $X \times_{\pi} D$ is regular, relatively minimal over $D$
	and has at least one singular fibre. Such an elliptic surface
	cannot be constant, as claimed.
\end{proof}

Applying  Theorem \ref{thm:Silverman} completes the proofs of Theorems \ref{thm:conic} and \ref{thm:non_reduced}. \qed

\section{Examples} \label{sec:examples}
\subsection{Rational elliptic surfaces} \label{sec:blow-ups}

 Let $F$ and $G$ be two plane cubic curves over $k$. Assume that $F$ and $G$ meet in at least one $k$-point and that $F$ is smooth. Let $X$ be the minimal desingularisation of the surface
\begin{equation} \label{eqn:pencil}
tF(x,y,z)+uG(x,y,z)=0 \quad \subset \PP^1 \times \PP^2.
\end{equation}
Then $X$ is a rational elliptic surface, with a common point $P$ of $F$ and $G$ giving a section over $k$. The projection onto $\PP^2$ realises $X$ as a blow-up of $\PP^2$ in the $9$ (possibly infinitely near) points over $\bar{k}$ which lie above the base locus $F \cap G$. The pencil of lines through $P$ gives rise to a conic bundle on $X$. Theorem \ref{thm:conic} thus applies.

\begin{corollary} \label{cor:rational}
	Let $r$ denote the generic rank of the elliptic fibration
	\eqref{eqn:pencil}.
	Then the set of curves with rank at least $r+2$ is not thin.
\end{corollary}

In the stated level of generality this result is new and improves on the results on rational elliptic surfaces from \cite{Billard, salgado1}. 

\begin{example} \hfill
	\begin{enumerate}
		\item Corollary \ref{cor:rational} applies to the 
		Mordell family $y^2 = x^3 + t$ (here $r = 0$).
		\item If the intersection of  $F$ and $G$ 
		consists of nine rational points in general position 
		then the elliptic fibration on $X$ has generic Mordell-Weil rank $8$.
		Corollary~\ref{cor:rational} implies that the collection
		of curves with rank at least $10$ is not thin.
		This construction generalises some of the constructions of elliptic
		curves of large rank given in \cite[\S11.2]{Ser97}.
	\end{enumerate}
\end{example}

\subsection{Some quadratic families of elliptic curves} \label{sec:DP1}
Consider families of elliptic curves of the form
\begin{equation} \label{eqn:KM}
y^2=a_3(t)x^3+a_2(t)x^2+a_1(t)x+a_0(t)
\end{equation}
where $\deg a_i(t) \leq 2$ and $a_3(t) \neq 0$. We assume that the family is non-constant and let $r$ be the generic rank. Using the  conic bundle structure given by projecting to $x$, in \cite[Cor.~2]{KM} it was shown that there are infinitely many $t \in \PP^1(k)$ for which the rank is positive. An application of Theorem \ref{thm:once} gives the following immediate improvement.

\begin{corollary}
	The set of $t \in k$ for which \eqref{eqn:KM} has rank at least $r+1$
	is not thin.
\end{corollary}

The quadratic twist family from Theorem \ref{thm:non_reduced} is a subfamily here. Once we remove this family, we obtain the following from Theorem \ref{thm:conic}.

\begin{corollary} 
	Assume that there is no polynomial $g(t)$ such that 
	$a_i(t)/g(t) \in k$ for all $i$.
	Then the set of $t \in k$ for which \eqref{eqn:KM} has rank at least $r+2$
	is not thin.
\end{corollary}

\subsection{Del Pezzo surfaces}
Let $X$ be a del Pezzo surface of degree $1$ with a conic bundle $\psi: X \to \PP^1$. Let $\widetilde{X}$ be the blow-up of $X$ in the base-point of the anticanonical linear system and $\pi: \widetilde{X} \to \PP^1$ the induced elliptic fibration. Every fibre of $\pi$ is geometrically integral, hence $\pi$ has generic rank $\rho(X) - 2$, where $\rho(X)$ is the Picard number of $X$. Theorem \ref{thm:conic} thus gives the 
following.

\begin{corollary} \label{cor:DP1}
	The set $\{ t \in \PP^1(k) : \rank \widetilde{X}_t(k) \geq \rho(X)\}$
	is not thin.
\end{corollary}

If $\psi$ is relatively minimal, then $X$ is  non-rational by \cite[Cor.~1.7]{Isk70}.

\end{document}